\documentclass{ifacconf}

\usepackage{natbib}

\usepackage{graphicx}
\graphicspath{{Figures/}}
\usepackage{subcaption}

\usepackage{mathtools, amssymb, bm}

\newtheorem{lemma}{Lemma}
\newtheorem{proposition}{Proposition}
\theoremstyle{definition}
\newtheorem{definition}{Definition}
\theoremstyle{remark}
\newtheorem{remark}{Remark}

\usepackage{array}
\usepackage{enumerate}
\usepackage{enumitem}
\usepackage{xcolor}
\usepackage{balance}
\usepackage{multirow}
\usepackage{booktabs}

\allowdisplaybreaks

\DeclareMathOperator*{\argmin}{arg\,min}
\DeclareMathOperator*{\sign}{sign}
\DeclareMathOperator*{\diag}{diag}
\DeclareMathOperator*{\U}{\mathcal{U}}
\DeclareMathOperator*{\Uc}{\mathcal{U}_c}
\DeclareMathOperator*{\Uuc}{\mathcal{U}_{uc}}

\makeatletter
\DeclareRobustCommand{\qed}{%
  \ifmmode 
  \else \leavevmode\unskip\penalty9999 \hbox{}\nobreak\hfill
  \fi
  \quad\hbox{\qedsymbol}}
\newcommand{\qedsymbol}{$\blacksquare$}
\newenvironment{proof}[1][\proofname]{\par
  \normalfont
  \topsep6\p@\@plus6\p@ \trivlist
  \item[\hskip\labelsep\itshape
    #1.]\ignorespaces
}{%
  \qed\endtrivlist
}
\newcommand{\proofname}{Proof}
\makeatother

\begin{document}
\begin{frontmatter}

\title{Energetic Resilience of Linear Driftless Systems\thanksref{footnoteinfo}} 

\thanks[footnoteinfo]{This work was supported by Air Force Office of Scientific Research grant FA9550-23-1-0131 and NASA University Leadership Initiative grant 80NSSC22M0070. (Corresponding Author: Ram Padmanabhan. Email: \texttt{ramp3@illinois.edu.}).}

\author{Ram Padmanabhan and} 
\author{Melkior Ornik}

\address{University of Illinois Urbana-Champaign, Urbana, IL 61801, USA.}

\begin{abstract}
When a malfunction causes a control system to lose authority over a subset of its actuators, achieving a task may require spending additional energy in order to compensate for the effect of uncontrolled inputs. To understand this increase in energy, we introduce an energetic resilience metric that quantifies the maximal additional energy required to achieve finite-time regulation in linear driftless systems that suffer this malfunction. 
We first derive optimal control signals and minimum energies to achieve this task in both the nominal and malfunctioning systems. We then obtain a bound on the worst-case energy used by the malfunctioning system, and its exact expression in the special case of loss of authority over one actuator. Further considering this special case, we derive a bound on the metric for energetic resilience. A simulation example on a model of an underwater robot demonstrates that this bound is useful in quantifying the increased energy used by a system suffering such a malfunction.
\end{abstract}


\end{frontmatter}

\section{Introduction}
Control systems can suffer from a variety of failures, either due to system faults or adversarial attacks. Such failures can prevent a system from achieving a specific performance objective, including reachability or safety. Informally, a system is said to be \emph{resilient} if {its} objectives can be achieved despite a failure. An example of a failure was the ``loss of attitude control'' \citep{Nauka} suffered by the Nauka module while docking to the International Space Station (ISS) in 2021. The failure mode in this event was that of \emph{partial loss of control authority}, where the module lost authority over some of its thrusters, and other thrusters on the ISS had to be used to counteract this uncontrolled thrust.

Motivated by this example, in this paper, we consider systems that are affected by complete loss in control authority over a subset of their actuators. In such a system, the \emph{uncontrolled} inputs, potentially chosen by an adversary, can take on any values in the input space. However, these uncontrolled inputs are \emph{measurable}, and can be used by the \emph{controlled} inputs --- under the authority of the system --- in order to achieve a task. 
We note that building actuator redundancy \citep{SP90, Grossman95} or using control reconfiguration schemes \citep{BFP19, HLXZ21} can aid resilience, but result in prohibitively high costs in control design. Moreover, standard fault-tolerant control considers only limited actuator failure modes \citep{TCJ02, AH19}, and adaptive and robust control strategies fail since the uncontrolled input is unrestricted \citep{BO22b, AD08}.

\cite{BO20} derived conditions for the \emph{resilient reachability} of linear systems, i.e., {reachability of a target set under any uncontrolled inputs}. {These conditions were used by \cite{BO22b} to design controllers that enable linear systems to achieve a target despite this malfunction.} To quantify the additional time it may take for a target to be reached under a loss of control authority, \cite{BXO21, BO23} introduced the notion of \emph{quantitative resilience} for linear driftless and general linear systems. {Quantitative resilience} was defined as the maximal ratio of minimum reach times to achieve a target between a nominal system and a malfunctioning system. 

In this paper, we introduce {a metric for a similar quantitative notion} called \emph{energetic resilience}. Instead of considering minimal reach times for the nominal and malfunctioning systems as in the works of \cite{BXO21} and \cite{BO23}, we consider the minimal control energies required to drive the state to the origin from an initial condition, in a given finite time. In practical systems, the maximal additional control energy required to achieve this task is directly related to maximal additional resource consumption, such as fuel in vehicles. 
Designing this resource capacity must take into account all possible malfunctioning inputs, which is not a straightforward problem. {Considering minimal control energies instead of minimal reach times} allows us to derive optimal control signals that achieve the task in both the nominal and malfunctioning driftless systems, which was not accomplished  by \cite{BXO21, BO22b, BO23}. We also derive a closed-form expression for the worst-case uncontrolled input that seeks to maximize the energy used by the malfunctioning system, which is not provided in the earlier works. {A similar metric was used to quantify the maximal \emph{cost of disturbance} by \cite{PBDO24}, considering linear systems affected by external bounded disturbances with no input constraints. In contrast, this paper considers bounded controlled inputs and uncontrolled inputs can take on values with the same magnitude.} As a preliminary step towards studying the resilience of general linear and nonlinear systems, we focus on linear driftless systems, which can characterize a variety of robotic and underwater systems \cite{YWX16}. 

Our contributions are organized as follows. In Section \ref{sec:Formulation}, we formulate the problem statement, providing the definitions of key quantities derived in this paper. Section \ref{sec:Energies} discusses the optimal control signals, minimal control energies and {restrictions on the time required} to achieve a task in both the nominal and malfunctioning systems. These results use a technical lemma characterizing the optimal control. We also derive an upper bound on the worst-case control energy for the malfunctioning system over all possible uncontrolled inputs. In Section \ref{sec:1Act}, we consider the special case of losing authority over one actuator, deriving an exact expression for the worst-case control energy for the malfunctioning system and the worst-case uncontrolled input achieving this energy. We also derive a bound on the metric for energetic resilience, quantifying how much additional control energy is used by the malfunctioning system compared to the nominal system. In Section \ref{sec:Example}, we consider {a model} of an underwater robot and illustrate the applicability of our results. In particular, we show that the resilience metric accurately characterizes the additional control energy used by the malfunctioning system.

\subsection{Notation and Facts} \label{sec:Notation}
The set $\mathbb{R}^+ \coloneqq [0, \infty)$ is the set of all non-negative real numbers. For scalars $z \in \mathbb{R}$, define the $\sign(\cdot)$ function as $\sign(z) = z/|z| \in \{-1, +1\}$ if $z \neq 0$, with $\sign(0) = 0$. The $\sign(\cdot)$ function operates elementwise on vectors $z \in \mathbb{R}^n$. The $p$-norm of a vector $x \in \mathbb{R}^n$ is defined as $\|x\|_p \coloneqq \left(\sum_{i=1}^{n}|x_i|^p\right)^{1/p}$, with $\|x\|_{\infty} \coloneqq \max_i |x_i|$. For a matrix $L \in \mathbb{R}^{p\times q}$ with entries indexed $l_{ij}$, define the induced matrix norms $\|L\|_1 \coloneqq \max_j \sum_{i=1}^{p} |l_{ij}|$ and $\|L\|_{\infty} \coloneqq \max_i \sum_{j=1}^{q} |l_{ij}|$. The Moore-Penrose inverse \citep{P55}, also called the pseudoinverse of $L$, is denoted $L^{\dagger}$. For a continuous function $u:[0, t_f]\to \mathbb{R}^p$, the $\mathcal{L}_2$ norm is defined as $\|u\|_{\mathcal{L}_2} \coloneqq \sqrt{\int_{t=0}^{t_f}\|u(t)\|_{2}^{2}~\mathrm{d}t}$.

For two matrices $M$ and $N$ such that the product $MN$ can be defined, the sub-multiplicative property of matrix norms is written as $\|MN\| \leq \|M\|\|N\|$, where $\|.\|$ is any induced norm. For any two vectors $x \in \mathbb{R}^n$ and $y \in \mathbb{R}^n$, the Cauchy-Schwarz inequality can be written as $|x^Ty| \leq \|x\|_2\|y\|_2$. The minimum and maximum eigenvalues of a symmetric matrix $P \in \mathbb{R}^{n\times n}$ are denoted $\lambda_{\min}(P)$ and $\lambda_{\max}(P)$. {For such a matrix, the Rayleigh inequality \citep{Horn} $\lambda_{\min}(P)\|x\|_{2}^{2} \leq x^TPx \leq \lambda_{\max}(P)\|x\|_{2}^{2}$ holds for any vector $x \in \mathbb{R}^n$.}

\section{Problem Formulation} \label{sec:Formulation}
In this paper, we consider linear driftless systems
\begin{equation} \label{eq:Nominal}
	\dot{x}(t) = Bu(t), ~~x(0) = x_0 \neq 0,
\end{equation}
where $x(t) \in \mathbb{R}^n$ is the state and $u(t) \in \mathbb{R}^{m+p}$ is the control. The set of admissible controls $\U$ is defined as:
\begin{equation} \label{eq:SetU}
\U \coloneqq \left\{u: \mathbb{R}^+ \to \mathbb{R}^{m+p}: \|u(t)\|_{\infty} \leq 1 ~\text{ for all $t$}\right\},
\end{equation}
in line with prior work \citep{BO23}. We consider malfunctions that result in the system losing control authority over $p$ of its $m+p$ actuators. Then, the matrix $B$ and control $u(t)$ can be split into \emph{controlled} and \emph{uncontrolled} components:
\begin{equation} \label{eq:Malfunctioning}
	\dot{x}(t) = \begin{bmatrix} B_c && B_{uc} \end{bmatrix} \begin{bmatrix} u_c(t) \\ u_{uc}(t) \end{bmatrix} = B_cu_c(t) + B_{uc}u_{uc}(t),
\end{equation}
where the subscripts $c$ and $uc$ denote \emph{controlled} and \emph{uncontrolled} respectively. Here, $B = [B_c, B_{uc}]$, $B_c \in \mathbb{R}^{n\times m}$, $B_{uc} \in \mathbb{R}^{n\times p}$, $u_c(t) \in \mathbb{R}^m$ and $u_{uc}(t) \in \mathbb{R}^p$. The set of admissible controls $\U$ splits into $\Uc$ and $\Uuc$ as
\begin{subequations}
\begin{align}
\Uc &\coloneqq \left\{u_c: \mathbb{R}^+ \to \mathbb{R}^{m}: \|u_c(t)\|_{\infty} \leq 1 ~\text{for all $t$}\right\}, \label{eq:SetUc} \\
\Uuc &\coloneqq \left\{u_{uc}: \mathbb{R}^+ \to \mathbb{R}^{p}: \|u_{uc}(t)\|_{\infty} \leq 1 ~\text{for all $t$}\right\}. \label{eq:SetUuc}
\end{align}
\end{subequations}
In this setting, the system has authority over the controlled input $u_c$, but the uncontrolled input $u_{uc}$ can be chosen arbitrarily from the space $\Uuc$, potentially by an adversary. However, $u_{uc}$ is {observable} and can be used to design $u_c$.

Our objective is the task of \emph{finite-time regulation}, achieving $x(t_f) = 0$ for a specified final time $t_f$ using the control $u \in \U$ in the nominal case and $u_c \in \Uc$ for any $u_{uc} \in \Uuc$ in the malfunctioning case. In this context, we define the notion of \emph{finite-time stabilizing resilience} of a system, adapted from \cite{BXO21}.

\begin{definition}[Finite-time Stabilizing Resilience] \label{def:Resilience}~
A system \eqref{eq:Nominal} is \emph{resilient} to the loss of control authority over $p$ of its actuators, represented by the matrix $B_{uc}$ if, for all uncontrolled inputs $u_{uc} \in \Uuc$ {and for a given final time $t_f$}, there exists a controlled input $u_c \in \Uc$ such that the system achieves finite-time regulation, i.e. $x(t_f) = 0$.
\end{definition}

{A system may be finite-time stabilizing resilient for some final times $t_f$ and not resilient for other final times. Throughout the rest of this paper, we refer to finite-time stabilizing resilience as simply \emph{resilience}.} We also assume that the nominal dynamics \eqref{eq:Nominal} are controllable so that finite-time regulation can be achieved from any initial state $x_0$. In both the nominal and malfunctioning cases, we are interested in the minimal control energies to achieve this task. Achieving this task might require considerably more control energy in the malfunctioning case compared to the nominal case. We thus aim to quantify the maximal additional control energy used by the malfunctioning system over all possible uncontrolled inputs $u_{uc}$, compared to the nominal system. To this end, we make the following definitions.

\begin{definition}[Nominal Energy] \label{def:EN}
The \emph{nominal energy} is the minimum energy in the input $u$ required to achieve finite-time regulation in time $t_f$ from the initial condition $x_0 \neq 0$, following the nominal dynamics \eqref{eq:Nominal}:
\begin{equation} \label{eq:EN_def}
	E_{N}^{*}(x_0, t_f) \coloneqq \inf_{u \in \U} \left\{\left\|u\right\|_{\mathcal{L}_2}^{2}~ \text{s.t. $x(t_f) = 0$ using \eqref{eq:Nominal}}\right\}.
\end{equation}
\end{definition}

\begin{definition}[Malfunctioning Energy] \label{def:EM}
The \emph{malfunctioning energy} is the minimum energy in the controlled input $u_c$ required to achieve finite-time regulation in time $t_f$ from the initial condition $x_0 \neq 0$, for a given uncontrolled input $u_{uc}$, following the malfunctioning dynamics \eqref{eq:Malfunctioning}:
\begin{align}
	E_{M}^{*}(x_0, t_f, u_{uc}) &\coloneqq \inf_{u_c(u_{uc}) \in \Uc} \Bigl\{\left\|u_c(u_{uc})\right\|_{\mathcal{L}_2}^{2} \nonumber \\
	&~\text{s.t. $x(t_f) = 0$ in \eqref{eq:Malfunctioning}, for the given $u_{uc}$}\Bigr\}, \label{eq:EM_def}
\end{align}
where $u_c(u_{uc})$ explicitly provides the dependence of the controlled input $u_c:\mathbb{R}^+ \to \mathbb{R}^m$ on the uncontrolled input $u_{uc}$, with a slight abuse of notation.
\end{definition}

\begin{definition}[Total Energy] \label{def:EMplus}
The \emph{total energy} is the energy used by \emph{all inputs} of the malfunctioning system \eqref{eq:Malfunctioning} for a given uncontrolled input $u_{uc}$, when the optimal controlled input $u_c(u_{uc})$ in \eqref{eq:EM_def} is used. In other words,
\begin{equation} \label{eq:EMplus}
E_{M}^{+}(x_0, t_f, u_{uc}) \coloneqq E_{M}^{*}(x_0, t_f, u_{uc}) + \left\|u_{uc}\right\|_{\mathcal{L}_2}^{2}.
\end{equation}
\end{definition}

\begin{definition}[Worst-case Total Energy] \label{def:EMbar}
The \emph{worst-case total energy} is the maximal effect of the uncontrolled input $u_{uc}$ on the total energy \eqref{eq:EMplus} used by the system:
\begin{align}
	\overline{E}_M(x_0, t_f) &= \sup_{u_{uc} \in \Uuc} \left\{E_{M}^{*}(x_0, t_f, u_{uc}) + \left\|u_{uc}\right\|_{\mathcal{L}_2}^{2}\right\}. \label{eq:EMbar_def}
\end{align}
\end{definition}

The goal of this paper is to understand how much larger the worst-case total energy is, compared to the nominal energy. This increase in energy is quantified by the following \emph{energetic resilience} metric.


\begin{definition}[Energetic Resilience] \label{def:rM}
For an initial condition $x_0$ at a distance of at least $R$ from the origin, we define the \emph{energetic resilience} of system \eqref{eq:Nominal} as
\begin{equation} \label{eq:rM_def}
r_M(t_f, R) \coloneqq \inf_{\|x_0\|_2 \geq R} \frac{E_{N}^{*}(x_0, t_f)}{\overline{E}_M(x_0, t_f)}.
\end{equation}
\end{definition}

Note that if $r_M(t_f, R) \geq 1/2$, then the malfunctioning system uses \emph{at most twice} the energy used by the nominal system to achieve finite-time regulation, for a given $t_f$ and $R$. We thus require a lower bound for this metric. Without the constraint $\|x_0\|_2 \geq R$, $E_{N}^{*}(x_0, t_f)$ can be arbitrarily close to zero, resulting in a trivial value for $r_M(t_f, R)$. We note that the definition of $r_M(t_f, R)$ is closely related to the definition of \emph{quantitative resilience} defined by \cite{BXO21} and \cite{BO23}. However, quantitative resilience was defined using reachable time while our definitions consider the control energy.

In the following section, we derive {closed-form expressions for the nominal and malfunctioning energy, as well as a bound on the worst-case total energy.} To quantify the maximal additional energy used by the malfunctioning system compared to the nominal system, we compare the quantities $E_{N}^{*}(x_0, t_f)$ from \eqref{eq:EN_def} and $\overline{E}_M(x_0, t_f)$ from \eqref{eq:EMbar_def}. This comparison is achieved in Section \ref{sec:1Act} by deriving bounds on the energetic resilience metric \eqref{eq:rM_def}, for the special case of losing authority over a single actuator.

\section{Nominal and Malfunctioning Energies} \label{sec:Energies}

In this section, we derive expressions for the energies defined in Section \ref{sec:Formulation}, as well as the corresponding optimal control inputs. We first present a lemma central to the development in this section.

\begin{lemma} \label{lem:MinEnergy}
Let $\mathcal{Z} \coloneqq \Bigl\{z:[0, t_f] \to \mathbb{R}^{n_z} : \frac{1}{t_f} \int_{0}^{t_f} z(t)\mathrm{d}t = \overline{z} \in \mathbb{R}^{n_z}\Bigr\}$ be the set of continuous, real vector-valued functions with given mean value $\overline{z}$ in the interval $[0, t_f]$. Let $z^*(t)$ denote the function with the minimum energy in $\mathcal{Z}$, i.e., $ z^* \coloneqq \argmin_{z \in \mathcal{Z}} \int_{0}^{t_f} z^T(t)z(t) \mathrm{d}t.$ Then, 
$$
z^*(t) = \overline{z} ~\text{ for all $t \in [0, t_f]$,}
$$
i.e., the function with minimum energy is a constant.
\end{lemma}
\begin{proof}
The result follows from the first-order condition for optimality using the Euler-Lagrange equation \citep{CVOC}, for the constrained calculus of variations problem
$$
\argmin_{z} \int_{0}^{t_f} z^T(t)z(t) \mathrm{d}t  ~\text{ s.t. }~ \frac{1}{t_f} \int_{0}^{t_f} z(t)\mathrm{d}t = \overline{z}.
$$
\end{proof}

Armed with this result, we first {derive a closed-form expression for the nominal energy.}

\subsection{Nominal Energy} \label{sec:Nominal}
Consider the system \eqref{eq:Nominal}. Solving for $x(t_f) = 0$,
\begin{equation} \label{eq:Nom_Sol1}
	x(t_f) = x_0 + \int_{0}^{t_f} Bu(t)\mathrm{d}t = 0.
\end{equation}
Define $\overline{u} \coloneqq \frac{1}{t_f} \int_{0}^{t_f} u(t)\mathrm{d}t$, the mean value of the control $u(t)$ in the interval $[0, t_f]$. Then, \eqref{eq:Nom_Sol1} reduces to
\begin{equation} \label{eq:Nom_Sol2}
	x_0 = -t_fB\overline{u}.
\end{equation}
Since we wish to find the minimum energy control signal $u$, we require the solution $\overline{u}$ with the least norm $\|\overline{u}\|_{2}$, based on Lemma \ref{lem:MinEnergy}. This \emph{least-norm} solution \citep[Corollary 2]{P56} is given by
\begin{equation} \label{eq:u_LS}
	\overline{u}^{LS} = -\frac{1}{t_f}B^{\dagger}x_0,
\end{equation}
where $B^{\dagger}$ is defined as in \cite{P55}. As the nominal dynamics \eqref{eq:Nominal} are assumed to be controllable, a control law achieves finite-time stabilization \emph{if and only if} its mean value over the interval $[0, t_f]$ is given by \eqref{eq:u_LS}. Next, note that the condition $\|\overline{u}^{LS}\|_{\infty} \leq 1$ is \emph{necessary} to ensure 
$u \in \U$. This condition is satisfied \emph{if and only if} 
\begin{equation} \label{eq:Nominal_Cond}
	t_f \geq \|B^{\dagger}x_0\|_{\infty}.
\end{equation}
Thus, condition \eqref{eq:Nominal_Cond} is {necessary} to ensure that a control $u$ achieving $x(t_f) = 0$ also satisfies $u \in \U$. From Lemma \ref{lem:MinEnergy}, over all signals with a given mean value, the minimum energy signal is a constant equal to that mean value. Thus, the signal achieving the infimum in \eqref{eq:EN_def}, denoted $u^*(t)$, is
\begin{equation} \label{eq:uN}
	u^*(t) = \overline{u}^{LS} = -\frac{1}{t_f}B^{\dagger}x_0 ~\text{ for all }~ t \in [0, t_f],
\end{equation}
from which
\begin{equation} \label{eq:EN}
	E_{N}^{*}(x_0, t_f) = \int_{0}^{t_f} u^{*T}(t)u^*(t)\mathrm{d}t = \frac{1}{t_f} \|B^{\dagger}x_0\|_{2}^{2}.
\end{equation}
Since $u^*(t) = \overline{u}^{LS}$, condition \eqref{eq:Nominal_Cond} is both necessary and sufficient to ensure $u^* \in \U$. 

\subsection{Malfunctioning Energy} \label{sec:Malfunctioning}
We now consider the malfunctioning system \eqref{eq:Malfunctioning} and derive expressions for the malfunctioning \eqref{eq:EM_def} and worst-case total \eqref{eq:EMbar_def} energies. Setting $x(t_f) = 0$ in the solution to \eqref{eq:Malfunctioning},
\begin{equation} \label{eq:Malf_Sol}
	x_0 + t_f B_c\overline{u}_c + t_f B_{uc}\overline{u}_{uc} = 0, 
\end{equation}
where $\overline{u}_c \coloneqq \frac{1}{t_f}\int_{0}^{t_f} u_c(t) \mathrm{d}t$ and $\overline{u}_{uc} \coloneqq \frac{1}{t_f} \int_{0}^{t_f} u_{uc}(t) \mathrm{d}t$ are the mean values of the control signals $u_c$ and $u_{uc}$. {Equation \eqref{eq:Malf_Sol} is linear} in the mean of the controlled input $\overline{u}_c$. {To find the minimum energy in the controlled input, based on Lemma \ref{lem:MinEnergy},} we are interested in the least-norm solution to this equation. {Analogously to \eqref{eq:u_LS}, the solution is given by}
\begin{equation} \label{eq:uc_LS}
	\overline{u}_{c}^{LS} = -\frac{1}{t_f} B_{c}^{\dagger}\left(x_0 + t_fB_{uc}\overline{u}_{uc}\right).
\end{equation}
Thus, finite-time stabilization can be achieved in the malfunctioning system \eqref{eq:Malfunctioning} if and only if the mean of the controlled input $u_c$ satisfies \eqref{eq:uc_LS}. Note that $\overline{u}_{c}^{LS}$ depends on the mean of the uncontrolled input $u_{uc}$ as well, which is an unknown quantity. {Further, constraints $u_c \in \Uc$ and $u_{uc} \in \Uuc$ reduce to $\|\overline{u}_{c}^{LS}\|_{\infty} \leq 1$ and $\|\overline{u}_{uc}\|_{\infty} \leq 1$. In the malfunctioning case, deriving a closed-form condition of the form of \eqref{eq:Nominal_Cond} is difficult, since condition $\|\overline{u}_{c}^{LS}\|_{\infty} \leq 1$ contains terms in $t_f$ in both its numerator and denominator and is clearly not linear in $t_f$.} However, for a given $t_f$, we can check whether $\|\overline{u}_{c}^{LS}\|_{\infty} \leq 1$ for all uncontrolled inputs ${u}_{uc}$ by checking the condition
\begin{equation} \label{eq:Malf_Cond}
	\max_{\|\overline{u}_{uc}\|_{\infty} \leq 1} \frac{1}{t_f} \left\|B_{c}^{\dagger}\left(x_0 + t_fB_{uc}\overline{u}_{uc}\right)\right\|_{\infty} \leq 1.
\end{equation}
This condition can easily be decomposed into a set of conditions that are linear in $\overline{u}_{uc}$, which can be checked efficiently on the vertices of the hypercube $\|\overline{u}_{uc}\|_{\infty} \leq 1$ \citep{LinOpt}. Finally, as a consequence of Lemma \ref{lem:MinEnergy}, we note that the control signal $u_{c}^{*}$ achieving the infimum in \eqref{eq:EM_def} is a constant equal to the mean value $\overline{u}_{c}^{LS}$. Thus, under the condition \eqref{eq:Malf_Cond} on $t_f$,
\begin{equation} \label{eq:uM}
	u_{c}^{*}(t) = \overline{u}_{c}^{LS} = -\frac{1}{t_f} B_{c}^{\dagger}\left(x_0 + t_fB_{uc}\overline{u}_{uc}\right) ~\text{for all}~ t \in [0, t_f].
\end{equation}
Clearly, the corresponding malfunctioning energy is
\begin{equation} \label{eq:EM}
	E_{M}^{*}(x_0, t_f, u_{uc}) = \frac{1}{t_f} \left\|B_{c}^{\dagger}\left(x_0 + t_fB_{uc}\overline{u}_{uc}\right)\right\|_{2}^{2}.
\end{equation}

We now aim to quantify the maximal effect of the uncontrolled input $u_{uc}$ on both the optimal controlled input $u_{c}^{*}$ in \eqref{eq:uM} and the malfunctioning energy $E_{M}^{*}(x_0, t_f, u_{uc})$ in \eqref{eq:EM}, using the worst-case total energy $\overline{E}_M(x_0, t_f)$ defined in \eqref{eq:EMbar_def}. A closed-form, analytical expression for the supremum in \eqref{eq:EMbar_def} using \eqref{eq:EM} is not straightforward to obtain, and hence we focus on bounding $\overline{E}_M(x_0, t_f)$.

\begin{proposition}[Worst-case Total Energy] \label{prop:EMbar}
The worst-case total energy $\overline{E}_M(x_0, t_f)$ is bounded from above as follows:
\begin{align}
\overline{E}_M(x_0, t_f) &\leq \frac{1}{t_f}\left\|B_{c}^{\dagger}x_0\right\|_{2}^{2} + t_f\left(\sum_{i = 1}^{p} \lambda_i \|V\|_{1}^{2} + p\right) \nonumber \\
&+ 2\|B_{uc}^{T}B_{c}^{\dagger T}B_{c}^{\dagger}x_0\|_1. \label{eq:EMbar}
\end{align}
\end{proposition}

\begin{proof}
Using \eqref{eq:EM}, we have
\begin{align}
&\overline{E}_M(x_0, t_f) = \sup_{u_{uc} \in \Uuc} \Biggl\{\frac{1}{t_f} \left\|B_{c}^{\dagger}\left(x_0 + t_fB_{uc}\overline{u}_{uc}\right)\right\|_{2}^{2} \nonumber \\
&+ \int_{0}^{t_f} u_{uc}^{T}(t)u_{uc}(t)\mathrm{d}t\Biggr\} \nonumber \\
&\leq \frac{1}{t_f}\left\|B_{c}^{\dagger}x_0\right\|_{2}^{2} + \underbrace{\sup_{\|\overline{u}_{uc}\|_{\infty} \leq 1} t_f \overline{u}_{uc}^{T}B_{uc}^{T}B_{uc}\overline{u}_{uc}}_{=T_1} \nonumber \\
&+ \underbrace{\sup_{u_{uc} \in \Uuc} 2x_{0}^{T}B_{c}^{\dagger T}B_{c}^{\dagger}B_{uc}\overline{u}_{uc}}_{=T_2} + \underbrace{\sup_{u_{uc} \in \Uuc} \int_{0}^{t_f} u_{uc}^{T}(t)u_{uc}(t)\mathrm{d}t}_{=T_3}, \label{eq:EMbar_1}
\end{align}
where the inequality is obtained on splitting the supremum over three different terms. {We first consider $T_1$.} Note that $B_{uc}^{T}B_{uc} \in \mathbb{R}^{p\times p}$ is positive semi-definite. Let $B_{uc}^{T}B_{uc} = V\Lambda V^T$ be its spectral decomposition where $\Lambda = \diag\{\lambda_1, \ldots, \lambda_p\}$ with each $\lambda_i \geq 0$ and $V$ is the {orthonormal} matrix of eigenvectors of $B_{uc}^{T}B_{uc}$. Further, let $q = V^T\overline{u}_{uc}$. Then, $\|q\|_{\infty} \leq \|V^T\|_{\infty} \|\overline{u}_{uc}\|_{\infty} \leq \|V\|_1$, where we use the sub-multiplicative property of norms, $\|V^T\|_{\infty} = \|V\|_1$ and $\|\overline{u}_{uc}\|_{\infty} \leq 1$ for $u_{uc} \in \Uuc$. Changing the optimization variable to $q$ using this property,
\begin{align}
T_1 &\leq t_f \sup_{\|q\|_{\infty} \leq \|V\|_1} q^T\Lambda q = t_f\sup_{\|q\|_{\infty} \leq \|V\|_1} \sum_{i = 1}^{p} \lambda_iq_{i}^{2} \nonumber \\
&= t_f\sum_{i = 1}^{p} \lambda_i \|V\|_{1}^{2}, \label{eq:T1}
\end{align}
where the optimal $q^*$ has each component $q_{i}^{*} = \pm \|V\|_1$, maximizing the objective since each $\lambda_i > 0$. 
Next, consider $T_2$. For any vector $\eta \in \mathbb{R}^p$, $\sup_{\|\overline{u}_{uc}\|_{\infty} \leq 1} \eta^T \overline{u}_{uc} = \|\eta\|_1,$ where the optimal $\overline{u}_{uc}$ is given by $\overline{u}_{uc}^{*} = \sign(\eta)$. Thus,
\begin{equation} \label{eq:T2}
T_2 = 2\|B_{uc}^{T}B_{c}^{\dagger T}B_{c}^{\dagger}x_0\|_1,
\end{equation}
and each component of $\overline{u}_{uc}^{*}$ that maximizes $T_2$ is either $+1$ or $-1$. This implies that each component of the worst-case uncontrolled input $u_{uc}^{*}(t)$ for $T_2$, also takes on a constant value of either $+1$ or $-1$ in the interval $[0, t_f]$. This {choice of $u_{uc}^{*}(t)$} also maximizes $T_3$, since the maximum possible value of the integrand $u_{uc}^{T}(t)u_{uc}(t) = p$ for $u_{uc} \in \Uuc$ is achieved at every time instant $t$, {where $p$ is the dimension of $u_{uc}$.} Hence, we have
\begin{equation} \label{eq:T3}
T_3 = t_fp.
\end{equation}
Substituting \eqref{eq:T1}, \eqref{eq:T2} and \eqref{eq:T3} in \eqref{eq:EMbar_1}, the result follows.
\end{proof}

\begin{remark}
The bound obtained in \eqref{eq:EMbar} may be conservative. {However, in Section \ref{sec:1Act}, we consider the special case when $p = 1$,} i.e., when control authority is lost over a single actuator, and derive a closed-form expression for $\overline{E}_{M}(x_0, t_f)$. We then use {that expression} to derive bounds on the energetic resilience metric \eqref{eq:rM_def}.
\end{remark}

\begin{remark} \label{rem:JB}
We note some similarities between the results of this section and the work of \cite{BXO21} on time-optimality for linear driftless systems. We have shown that the energy-optimal control signals in the nominal \eqref{eq:uN} and malfunctioning \eqref{eq:uM} cases are constant. Similarly, it was shown by \cite{BXO21} that the \emph{time}-optimal control signals in both the nominal and malfunctioning cases are constant; however, closed-form expressions for this signal are not provided by \cite{BXO21}, unlike in \eqref{eq:uN} and \eqref{eq:uM}. We also note in the malfunctioning case that each component of the worst-case uncontrolled input $u_{uc}^{*}(t)$ takes on a constant value of either $+1$ or $-1$ in the interval $[0, t_f]$. A similar result in \citep{BXO21} shows that the worst-case uncontrolled input maximizing reachable time also has constant components with maximum allowable amplitude. However, closed-form expressions for this worst-case uncontrolled input are also not provided by \cite{BXO21}.
\end{remark}


\section{Resilience to the Loss of One Actuator} \label{sec:1Act}
We now consider the special case of losing control authority over one actuator, i.e., $p = 1$. We derive a closed-form expression for the worst-case total energy \eqref{eq:EMbar_def} {and the corresponding} worst-case uncontrolled input. Using this expression, we derive a bound on the resilience metric \eqref{eq:rM_def}.

\subsection{Worst-case Total Energy}
When $p = 1$, $B_{uc} \in \mathbb{R}^n$ and $u_{uc}(t) \in \mathbb{R}$. From \eqref{eq:EMbar_def} and \eqref{eq:EM},
\begin{align}
&\overline{E}_M(x_0, t_f) \nonumber \\
&= \sup_{u_{uc} \in \Uuc} \Biggl\{\frac{1}{t_f} \left\|B_{c}^{\dagger}\left(x_0 + t_fB_{uc}\overline{u}_{uc}\right)\right\|_{2}^{2} + \int_{0}^{t_f} u_{uc}^{2}(t)\mathrm{d}t\Biggr\} \nonumber \\
&= \frac{1}{t_f}\left\|B_{c}^{\dagger}x_0\right\|_{2}^{2} + \sup_{u_{uc} \in \Uuc} \Biggl\{t_fB_{uc}^{T}B_{uc} \overline{u}_{uc}^{2} \nonumber \\
&+ 2x_{0}^{T}B_{c}^{\dagger T}B_{c}^{\dagger}B_{uc}\overline{u}_{uc} + \int_{0}^{t_f} u_{uc}^{2}(t)\mathrm{d}t\Biggr\}.\label{eq:EMbar_2}
\end{align}
Recall that $u_{uc} \in \Uuc$ implies $\|\overline{u}_{uc}\|_{\infty} \leq 1$, or $-1 \leq \overline{u}_{uc} \leq 1$. Note that the second term in the supremum is maximized by $\overline{u}_{uc}^{*} = \sign(B_{uc}^{T}B_{c}^{\dagger T}B_{c}^{\dagger}x_0)$ as in Proposition \ref{prop:EMbar}, and reduces to $2\left|B_{uc}^{T}B_{c}^{\dagger T}B_{c}^{\dagger}x_0\right|$. When $\sign(B_{uc}^{T}B_{c}^{\dagger T}B_{c}^{\dagger}x_0) \in \{-1,+1\}$, this choice of $\overline{u}_{uc}^{*}$ also maximizes the other terms inside the supremum so that
\begin{align}
\overline{E}_M(x_0, t_f) &= \frac{1}{t_f}\left\|B_{c}^{\dagger}x_0\right\|_{2}^{2} + t_f\left(\|B_{uc}\|_{2}^{2} + 1\right) \nonumber \\
&+ 2\left|B_{uc}^{T}B_{c}^{\dagger T}B_{c}^{\dagger}x_0\right|. \label{eq:EMbar_1Act}
\end{align}
Alternatively, if $B_{uc}^{T}B_{c}^{\dagger T}B_{c}^{\dagger}x_0 = 0$, the second term in the supremum vanishes, reducing the value of $\overline{E}_M(x_0, t_f)$. The worst-case total energy is thus achieved only when $B_{uc}^{T}B_{c}^{\dagger T}B_{c}^{\dagger}x_0 \neq 0$, with the constant uncontrolled input 
\begin{equation} \label{eq:uuc}
u_{uc}^{*}(t) = \sign(B_{uc}^{T}B_{c}^{\dagger T}B_{c}^{\dagger}x_0) ~\text{ for all }~ t \in [0, t_f].
\end{equation}
%
%

\subsection{Energetic Resilience Metric} \label{sec:Metric}
We now consider the energetic resilience metric $r_M(t_f, R)$ in \eqref{eq:rM_def}, and use the expression \eqref{eq:EMbar_1Act} to derive bounds on the this metric in the case when control authority is lost over a single actuator. We note that a bound on $r_M(t_f, R)$ can be derived when control authority is lost over $p > 1$ actuators using \eqref{eq:EMbar}. However, this bound is more conservative since \eqref{eq:EMbar} is only an upper bound on the worst-case total energy and not an exact expression as in \eqref{eq:EMbar_1Act}.

\begin{proposition} \label{prop:rM}
When $p = 1$, the energetic resilience metric $r_M(t_f, R)$ is bounded from below as follows:
\begin{align}
&r_M(t_f, R) \nonumber \\
&\geq \frac{R^2\lambda_{\min}\left(B^{\dagger T}B^{\dagger}\right)}{R^2L + 2Rt_f\|B_{uc}^{T}B_{c}^{\dagger T}B_{c}^{\dagger}\|_2 + t_{f}^{2}\left(\|B_{uc}\|_{2}^{2} + 1\right)}, \label{eq:rM_bound}
\end{align}
where $L \coloneqq \lambda_{\max}\left(B_{c}^{\dagger T}B_{c}^{\dagger}\right)$.
\end{proposition}

\begin{proof}
We provide a brief sketch of the proof to conserve space. Using \eqref{eq:EN}, \eqref{eq:EMbar_1Act} and \eqref{eq:rM_def},
\begin{align}
&r_M(t_f, R) =  \nonumber \\
&\inf_{\|x_0\|_2 \geq R} \frac{\frac{1}{t_f}\|B^{\dagger}x_0\|_{2}^{2}}{\frac{1}{t_f}\|B_{c}^{\dagger}x_0\|_{2}^{2} + 2\left|B_{uc}^{T}B_{c}^{\dagger T}B_{c}^{\dagger}x_0\right| + t_f\left(\|B_{uc}\|_{2}^{2} + 1\right)} \nonumber \\
&= \nonumber \\
&\frac{1}{\sup_{\|x_0\|_2 \geq R} \left\{\frac{\|B_{c}^{\dagger}x_0\|_{2}^{2}}{\|B^{\dagger}x_0\|_{2}^{2}} + 2t_f\frac{\left|B_{uc}^{T}B_{c}^{\dagger T}B_{c}^{\dagger}x_0\right|}{\|B^{\dagger}x_0\|_{2}^{2}} + t_{f}^{2}\frac{\|B_{uc}\|_{2}^{2} + 1}{\|B^{\dagger}x_0\|_{2}^{2}}\right\}}. \label{eq:rM_1}
\end{align}
Let $L \coloneqq \lambda_{\max}\left(B_{c}^{\dagger T}B_{c}^{\dagger}\right)$. Using the Rayleigh inequality on $B_{c}^{\dagger T}B_{c}^{\dagger}$ and $B^{\dagger T}B^{\dagger}$ and the Cauchy-Schwarz inequality, each term in the supremum can be bounded as follows:
\begin{align*}
\sup_{\|x_0\|_2 \geq R} \frac{\|B_{c}^{\dagger}x_0\|_{2}^{2}}{\|B^{\dagger}x_0\|_{2}^{2}} &\leq \frac{L}{\lambda_{\min}\left(B^{\dagger T} B^{\dagger}\right)}, \\
\sup_{\|x_0\|_2 \geq R} 2t_f\frac{\left|B_{uc}^{T}B_{c}^{\dagger T}B_{c}^{\dagger}x_0\right|}{\|B^{\dagger}x_0\|_{2}^{2}} &\leq 2t_f \frac{\|B_{uc}^{T}B_{c}^{\dagger T}B_{c}^{\dagger}\|_2}{R\lambda_{\min}\left(B^{\dagger T}B^{\dagger}\right)},\\
\sup_{\|x_0\|_2 \geq R} t_{f}^{2}\frac{\|B_{uc}\|_{2}^{2} + 1}{\|B^{\dagger}x_0\|_{2}^{2}} &\leq t_{f}^{2}\frac{\|B_{uc}\|_{2}^{2} + 1}{R^2\lambda_{\min}\left(B^{\dagger T}B^{\dagger}\right)}.
\end{align*}
Substituting in \eqref{eq:rM_1}, the result follows.
\end{proof}

In the following section, we present an example demonstrating the use of this metric in quantifying the maximal additional energy used by the malfunctioning system compared to the nominal system.

\section{Simulation Example} \label{sec:Example}
We now illustrate the use of the energetic resilience metric on a driftless model of an underwater robot. We consider a minor modification of the model considered by \cite{BO20}:
\begin{equation} \label{eq:Model}
\begin{bmatrix} \dot{x} \\ \dot{y} \end{bmatrix} = \begin{bmatrix} 2 & \phantom{-}1 & \phantom{-}1 \\ 0.2 & -1 & \phantom{-}1 \end{bmatrix} \begin{bmatrix} u_1 \\ u_2 \\ u_3 \end{bmatrix},
\end{equation}
where the input matrix is $B$. Here, $x$ and $y$ are the coordinates of the robot. The engine $u_1$ acts primarily along the $x$-direction, with a small bias in the $y$-direction. $u_2$ and $u_3$ are engines at a $45^{\circ}$ angle to the $x$- and $y$-directions, opposing each other. {The model in \cite{BO20} heavily biased the engine $u_1$ towards the $x$-direction, and the corresponding entry in the input matrix was $10$. We reduce this bias by modifying the entry to $2$.} 

We assume that a malfunction causes the system to lose authority over $u_3$. Then, $B_c = \begin{bmatrix} 2 & \phantom{-}1 \\ 0.2 & -1 \end{bmatrix} ~\text{ and }~  B_{uc} = \begin{bmatrix} 1 \\ 1 \end{bmatrix}.$ In what follows, for notational simplicity, we denote $E_N \equiv E_{N}^{*}(x_0, t_f)$, $\overline{E}_M \equiv \overline{E}_M(x_0, t_f)$ and $E_{M}^{+} \equiv E_{M}^{+}(x_0, t_f, u_{uc})$. We now compare the nominal energy $E_N$ \eqref{eq:EN} to the total energy $E_{M}^{+}$ from \eqref{eq:EM} and the definition \eqref{eq:EMplus}, and worst-case total energy $\overline{E}_M$ \eqref{eq:EMbar_1Act}. In particular, we demonstrate the accuracy of the energetic resilience bound \eqref{eq:rM_bound} in quantifying this comparison. We fix $t_f = 10$ and vary the distance $R$ of the initial condition $x_0$ from the origin. 

\begin{figure}[!t]
	\centering
	\includegraphics[width = 0.45\textwidth]{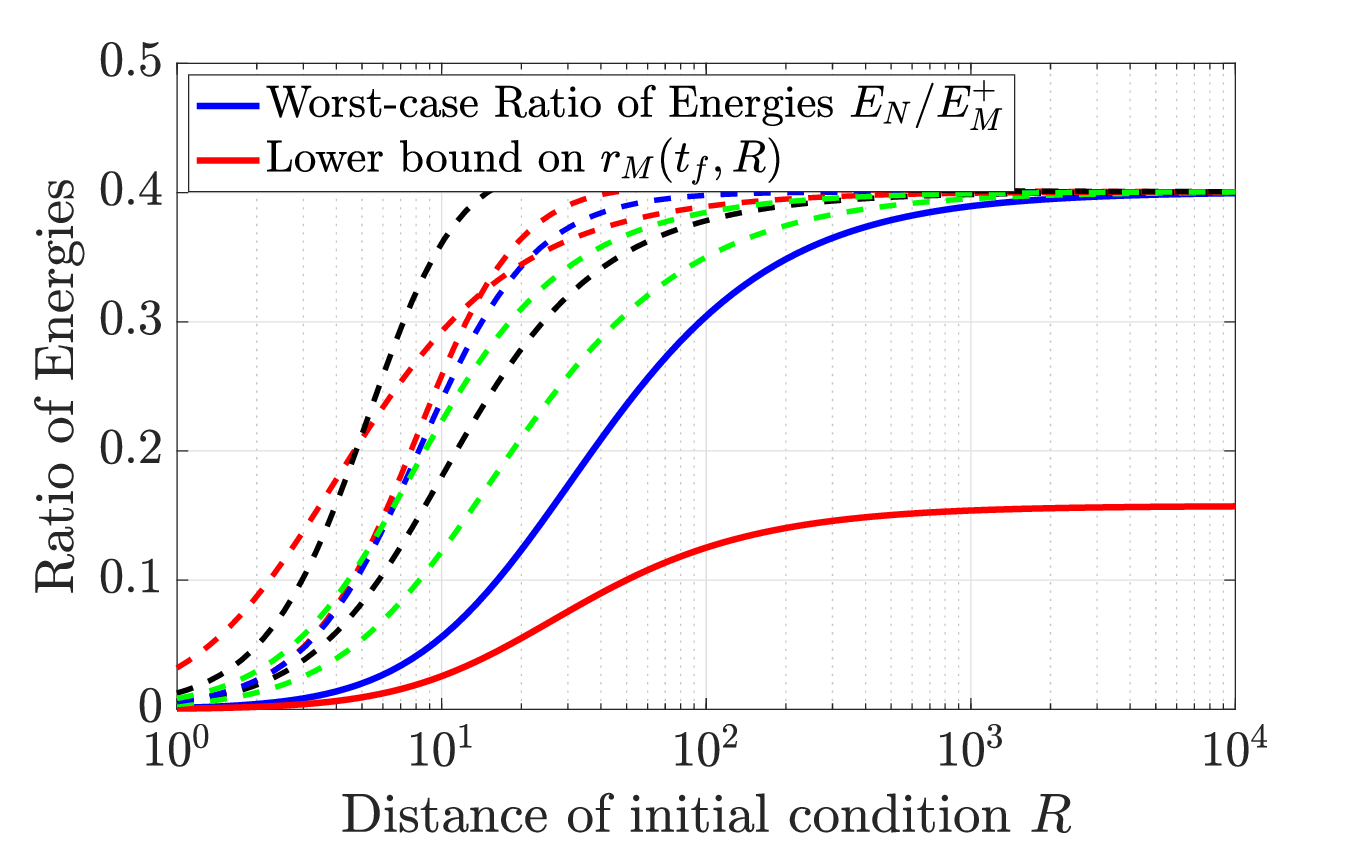}
	\vspace{-0.2cm}
	\caption{Ratio of nominal and malfunctioning energies compared to the resilience metric. The thinner dashed lines plot $E_{M}^{+}-E_N$ for various uncontrolled inputs.}
	\label{fig:Metric}
\end{figure}

Fig.~\ref{fig:Metric} shows the ratio of the nominal energy $E_N$ and the total energy $E_{M}^{+}$ as a function of $R$ in the thin dashed lines. Here, we test a large class of uncontrolled inputs, including constant, full-amplitude inputs of the form \eqref{eq:uuc} and different classes of low- and high-frequency sinusoids. We also plot the lower bound on $r_M(t_f, R)$ from \eqref{eq:rM_bound} as a function of $R$, shown in the thick red line. It is evident that $r_M(t_f, R)$ bounds the ratio of nominal and augmented malfunctioning energies from below. Furthermore, the thick blue line represents the ratio $E_{N}/E_{M}^{+}$ for the worst-case uncontrolled input in \eqref{eq:uuc}, where $E_{M}^{+} = \overline{E}_M$. This ratio is calculated using \eqref{eq:EN} and \eqref{eq:EMbar_1Act}. The bound on $r_M(t_f, R)$ is a reasonable approximation for this ratio, especially for initial conditions closer to the origin. For instance, when $R = 10$, we have $r_M(t_f, R) \approx 0.03$, while $E_N/\overline{E}_M \approx 0.05$. The bound on $r_M(t_f, R)$ indicates that \emph{at most} $1/0.03 = 33.33$ times more energy is used by the actuators to achieve finite-time regulation when a loss of control authority of the form described above occurs. The actual ratio indicates that around $1/0.05 \approx 20$ times the energy is required to achieve the task. The metric $r_M(t_f, R)$ is thus a useful quantity to characterize the maximal additional energy required to achieve a task when control authority is lost over a subset of actuators.

If we considered the model used by \cite{BO20}, this metric would be relatively less useful in characterizing the additional energy required, in comparison to the model in \eqref{eq:Model}. The heavier weight of $u_1$ on $x$ contributes to poor conditioning of the matrices $B^{\dagger T}B^{\dagger}$ and $B_{c}^{\dagger T} B_{c}^{\dagger}$, resulting in a more conservative bound in \eqref{eq:rM_bound} than in the example presented here.

\section{Conclusions and Future Work} \label{sec:Conclusion}
In this paper, our objective was to quantify the maximal additional energy used by a system which loses control authority over a subset of actuators, compared to a system with no such malfunction. We thus considered the special case of linear driftless systems and introduced an energetic resilience metric comparing the nominal and worst-case total energies to achieve finite-time regulation. Deriving the nominal and worst-case total energies for this task used a technical lemma proved using the calculus of variations. When considering the special case of losing control authority over one actuator, we obtained an exact expression for the worst-case total energy, allowing us to obtain a bound on the resilience metric. A simulation example on a model of an underwater robot demonstrated the applicability of this metric in characterizing the additional energy used by a malfunctioning system. Our ongoing work involves examining resilience of general nonlinear systems.

\balance

\bibliography{references}

\end{document}